\documentclass[reqno,centertags,12pt]{amsart}
\usepackage{amsmath,amsthm,amscd,amssymb,latexsym,verbatim}

\usepackage{graphicx,epsf}

-\marginparwidth \oddsidemargin -4mm \evensidemargin \oddsidemargin

\newtheorem{theorem}{Theorem}[section]

\newtheorem{lemma}[theorem]{Lemma}
\newtheorem{corollary}[theorem]{Corollary}
\theoremstyle{definition}

\theoremstyle{remark}

\newcounter{smalllist}

\DeclareMathOperator{\diam}{diam}

\allowdisplaybreaks
\numberwithin{equation}{section}

\newcommand{\lb}{\label}

\newcommand{\beq}{\begin{equation}}
\newcommand{\eeq}{\end{equation}}

\newcommand{\bal}{\begin{align}}
\newcommand{\eal}{\end{align}}
\newcommand{\bals}{\begin{align*}}
\newcommand{\eals}{\end{align*}}

\newcommand{\bbN}{{\mathbb{N}}}
\newcommand{\bbR}{{\mathbb{R}}}

\newcommand{\bbP}{{\mathbb{P}}}
\newcommand{\bbE}{{\mathbb{E}}}
\newcommand{\bbZ}{{\mathbb{Z}}}

\newcommand{\bbT}{{\mathbb{T}}}

\newcommand{\calB}{{\mathcal B}}
\newcommand{\calI}{{\mathcal I}}

\newcommand{\eps}{\varepsilon}
\newcommand{\del}{\delta}
\newcommand{\tht}{\theta}

\newcommand{\til}{\tilde}

\newcommand{\ce}{c^*_e}
\newcommand{\diff}{D}
\newcommand{\diffe}{\diff_e}
\newcommand{\Om}{\Omega}
\newcommand{\tilOm}{{\tilde\Omega}}

\begin{document}
\title[Reaction-Diffusion Front Speed Enhancement]
{Reaction-Diffusion Front Speed Enhancement by Flows}

\author{Andrej Zlato\v s}

\address{\noindent Department of Mathematics \\ University of
Chicago \\ Chicago, IL 60637, USA \newline Email: \tt
zlatos@math.uchicago.edu}


\begin{abstract}
We study flow-induced enhancement of the speed of pulsating traveling fronts for reaction-diffusion equations, and quenching of reaction by fluid flows. We prove, for periodic flows in two dimensions and any combustion-type reaction, that the front speed is proportional to the square root of the (homogenized) effective diffusivity of the flow. We show that this result does not hold in three and more dimensions. We also prove conjectures from \cite{ABP,Berrev,FKR} for cellular flows, concerning the rate of speed-up of fronts and the minimal flow amplitude necessary to quench solutions with initial data of a fixed (large) size.
\end{abstract}

\maketitle

\section{Introduction and the Main Results} \lb{S1}

It is well known that the presence of a fluid flow can significantly increase mixing properties of diffusion.  The study of this phenomenon, sometimes called {\it eddy diffusivity}, has been the aim of a large body of mathematical and physical literature. Questions of  long time--large scale behavior are usually addressed via techniques of homogenization theory (see, e.g., \cite{JKO,KM} and references therein). This approach is appropriate when one can wait a long time for mixing to take effect. The presence of other processes, however, may introduce additional time scales to the problem. One such process is reactive combustion, which happens on short time scales and therefore requires a different approach to the study of combustive mixing.

The effects of flows on combustion have recently been studied by various authors, both qualitatively and quantitatively \cite{ABP,Berrev,BHN-2,CKOR,CKR,CKRZ,FKR,Heinze,KRS,KS,KR,KZ,NR,RZ,VCKRR,ZlaArrh,ZlaPersym,ZlaMix,ZlaPercol}. The main effects are two-fold. A strong flow can speed up propagation of a reaction (such as a wind spreading a fire) but also extinguish it (the ``try to light a match in the wind'' effect). The models used are reaction-advection-diffusion equations, in which the first phenomenon is manifested by the enhancement of speed of their (pulsating) traveling front solutions, and the second by quenching of solutions with (large) compactly supported initial data.

In the present paper we study both these effects for stationary periodic flows. We consider the reaction-advection-diffusion equation
\beq \lb{1.1}
T_t + u(x)\cdot\nabla T = \Delta T + f(T)
\eeq
for the (normalized) temperature $T(t,x)\in[0,1]$ of  a premixed combustible gas, with $(t,x)\in \bbR\times \bbR^d$. The gas is advected by a periodic incompressible (i.e., $\nabla \cdot u\equiv 0$)  mean-zero vector field $u\in C^{1,\del}(\bbR^d)$ (also calld {\it flow}). For the sake of simplicity, we will assume that all periods of $u$ are 1, that is, the cell of periodicity is the $d$-dimensional torus $\bbT^d\equiv[-\tfrac 12,\tfrac 12]^d$, with $-\tfrac 12$ and $\tfrac 12$ identified. The general periodic case is identical. The non-negative {\it reaction function} $f\in C^{1,\del}([0,1])$ accounts for the increase of temperature due to a chemical reaction such as burning, and is of the {\it combustion type}. That is, there is $\tht\in[0,1)$ such that $f(s)=0$ for $s\in[0,\tht]\cup\{1\}$,  $f(s)>0$ for $s\in(\tht,1)$, and $f$ is non-increasing on $(1-\del,1)$ for some $\del>0$. If $\tht>0$, then $f$ is an {\it ignition reaction} (with {\it ignition temperature} $\tht$), otherwise $f$ is a {\it positive reaction}. A special case of the latter is the {\it Kolmogorov-Petrovskii-Piskunov (KPP) reaction} \cite{KPP} with $0<f(s)\le sf'(0)$ for all $s\in(0,1)$.

A {\it pulsating traveling front} in the direction of a unit vector $e\in\bbR^d$ is a solution of \eqref{1.1} of the form $T(t,x)=U(x\cdot e-ct,x)$, with $c\in\bbR$ the {\it front speed} and $U:\bbR\times\bbR^d\to[0,1]$ periodic in the second variable such that
\beq \lb{1.2}
\lim_{s\to-\infty}U(s,x)=1 \qquad\text{and}\qquad
\lim_{s\to+\infty}U(s,x)=0,
\eeq
uniformly in $x\in\bbR^d$. It is well known that under our assumptions on $u$ and $f$, for each $e\in\bbR^d$ there is a unique $\ce(u,f)>0$ such that a pulsating traveling front in direction $e$ and with speed $c$ exists if and only if $c=\ce(u,f)$ for ignition reactions \cite{Xin3}, resp. $c\in[\ce(u,f),\infty)$ for positive reactions \cite{BH}. 

In both cases we will be interested in the fronts with the {\it unique/minimal speeds} $\ce(u,f)$. These are the most physical ones because they determine the speed of spreading of solutions to the Cauchy problem for \eqref{1.1} with (large enough) compactly supported initial data. 
An exact formula for $\ce(u,f)$ has been obtained for general reactions in \cite{ElSmaily} (and for the special case of KPP reactions also earlier in \cite{BHN-1}). Unfortunately, it is a complicated variational expression and it is not obvious how to use it to obtain simple general estimates on $\ce(u,f)$. 

Our goal is to derive simple estimates on the front speed $\ce(u,f)$, and use them to answer open questions from \cite{ABP,Berrev,FKR} concerning speed-up of pulsating fronts and quenching of reaction by strong flows. We will provide such estimates in two dimensions, in terms of the size of $f$ and the {\it effective diffusivity} $\diffe(u)$ of the flow $u$ (in the direction $e$). The latter quantity can be found in a much simpler way than $\ce(u,f)$ using \eqref{1.5} and \eqref{1.6} below. It appears in the homogenization theory which, as mentioned above, is applicable to the study of long time behavior of the solutions of the related linear PDE
\beq \lb{1.3}
\psi_t+u(x)\cdot\nabla \psi=\Delta \psi.
\eeq
Despite the fact that the presence of reaction introduces a short time scale to the model, we will be able to show by other methods that in two dimensions (but not in three!), the effective diffusivity determines the front speed up to a bounded factor.

The long time behavior of the solutions of \eqref{1.3} is governed by the {\it effective diffusion equation}
\beq\label{1.4}
\bar \psi_t = \nabla \cdot(\diff(u) \nabla  \bar \psi),
\eeq
where the ($x$-independent positive symmetric) {\it effective diffusivity matrix} $\diff(u)$ is
obtained as follows. For any $e\in\bbR^d$, let $\chi_e(x)$ be the
periodic mean-zero solution of the cell problem
\beq\label{1.5}
-\Delta\chi_e+u\cdot\nabla\chi_e=u\cdot e
\eeq
on $\bbT^d$. Then $\diff(u)$ is given by
\[
e\cdot\diff(u)e' = \int_{\bbT^d}(\nabla\chi_e + e)\cdot(\nabla\chi_{e'} + e') dx 
= e\cdot e' + \int_{\bbT^d}\nabla\chi_e\cdot\nabla\chi_{e'} dx
\]
for any $e,e'\in\bbR^d$. The effective spreading in the direction of a unit vector $e\in\bbR^d$
is now governed by the {\it effective diffusivity}
\beq\label{1.6}
\diffe(u)\equiv e\cdot\diff(u)e = 1+\|\nabla\chi_e\|^2_{L^2(\bbT^d)} .
\eeq

When the nonlinearity in \eqref{1.1} is weak (the reaction time scale is large) so that we have
\[
T_t+u(x)\cdot\nabla T=\Delta T+\eps f(T)
\]
with $\eps\ll 1$, the long time--large space scaling $t\mapsto t/\eps^2$, $x\mapsto x/\eps$ gives
\[
T_t+\frac{1}{\eps}u\left(\frac{x}{\eps}\right)\cdot\nabla T=\Delta T+f(T).
\]
The homogenized version of this equation is
\[
\bar T_t=\nabla\cdot(\diff(u)\nabla \bar T)+f(\bar T),
\]
with the unique/minimal front speed in direction $e$ depending on $\diffe(u)$ and $f$.  Although the above approximation holds only on certain space--time scales, it does suggest a relation between $\ce(u,f)$ on one hand and $\diffe(u)$ and $f$ on the other. Our main result confirms this relation in two dimensions:

\begin{theorem}\lb{T.1.1}
Let $f$ be any combustion-type reaction. There are $C_1(f), C_2(f)>0$ such that for any 1-periodic incompressible mean-zero $C^{1,\del}$ flow $u$ on $\bbR^2$ and any unit vector $e\in\bbR^2$,
\beq\lb{1.7}
C_1(f) \sqrt{\diffe(u)} \le \ce(u,f) \le C_2(f) \sqrt{\diffe(u)}.
\eeq
\end{theorem}

{\it Remarks.}  1. We have the following estimates on the constants in \eqref{1.7}. From the results of \cite{RZ} and monotonicity of $\ce(u,f)$ in $f$ it follows that 
\beq \lb{1.7a}
C_2(f)\le C \sqrt{\left\| f(s)/s \right\|_\infty} \left(1+ \sqrt{\left\| f(s)/s \right\|_\infty} \right)
\eeq
for some $C>0$.  If $m_\zeta(f)=\min\{f(s)\,|\, s\in[\zeta,1-\tfrac{(1-\zeta)^2}8]\}>0$ (for each $f$ as above there is such $\zeta\in(0,1)$), then from the proof of Theorem \ref{T.1.1} it follows that
 \beq \lb{1.7b}
 C_1(f)\ge C_\zeta \frac{\sqrt{m_\zeta(f)}} {1+ \sqrt{m_\zeta(f)}}
 \eeq
for some $C_\zeta>0$. We note that this estimate is optimal up to a constant for small $m_\zeta(f)$ (due to \eqref{1.7a}), but also for large $m_\zeta(f)$. Indeed,  $C_1(f)$ is uniformly bounded above in $f$ for KPP reactions \cite{RZ} and thus for all reactions.
 \smallskip
 
 2. In particular, $C_j(mf)\sim \sqrt m$ ($j=1,2$) as $m\to 0$ for any fixed $f$. 
\smallskip

3. In the case of KPP reactions, this result has been proved in \cite{NR,RZ}. This case is considerably simpler due to the fact that $\ce(u,f)=\ce(u,\til f)$, where $\til f(s)\equiv f'(0)s$ and $\ce(u,\til f)$ is the minimal front speed for the {\it linear equation} \eqref{1.1} with $\til f$ in place of $f$. (Fronts for $\til f$ do not converge to 1 as $x\cdot e\to -\infty$ but rather grow exponentially.)
\smallskip


The relation $\ce(u,f)\sim C(f)\sqrt{\diffe(u)}$ is analogous to that in the case of $u\equiv 0$ and constant diffusivity matrix $\diff>0$. Indeed, spatial scaling $x\mapsto\sqrt\diff x$ shows that the unique/minimal front speed in the direction $e$ for $T_t=\nabla\cdot(\diff\nabla T)+f(T)$ is $C(f)\sqrt{e\cdot\diff e}$ (with $C(f)\equiv \ce(0,f)$ independent of $e$). The problem with $u\not\equiv 0$, however, is that the convergence of solutions of \eqref{1.3} to those of \eqref{1.4} occurs on large time scales, while solutions of \eqref{1.1} can be effectively estimated using \eqref{1.3} on short time scales only.  We will therefore study {\it short time diffusivity} for \eqref{1.3} in Section \ref{S2} and relate it to $\diffe(u)$ in Theorem \ref{T.2.1} below. The theorem, which applies in all dimensions, will be a key step towards overcoming this problem for $d=2$.

The restriction to two dimensions comes from Theorem \ref{T.3.1} below, a relation between the speed and the width of a front. It turns out, in fact, that this result and Theorem \ref{T.1.1} are false in dimensions $d\ge 3$, and Theorem \ref{T.5.1} shows that the bounds in \eqref{1.7} cannot hold with flow-independent $C_1(f)$ and $C_2(f)$ for $d\ge 3$!

This raises an interesting question about large deviations of the stochastic process $X_t^x$ from \eqref{2.1}, corresponding to \eqref{1.3}. If $I_{e,u}$ is the rate function for $Z_t^{x,e}\equiv (x-X_t^x)\cdot e$ (i.e., $\lim_{t\to\infty} t^{-1} \ln \bbP_\Omega(Z_t^{x,e}>ct)=-I_{e,u}(c)$ for $c>0$ and any $x$), then $I_{e,u}(\ce(u,f))=f'(0)$ holds for KPP reactions. From \cite{RZ} we know that in two dimensions we have $\ce(u,f)/\sqrt{\diffe(u)} = 2\sqrt{f'(0)} + O(f'(0)^{3/4})$ for small $f'(0)$ and KPP $f$, with an $(e,u)$-independent error bound. This means that  $I_{e,u}(c)\approx c^2/4\diffe(u)$ for $c\lesssim \sqrt{\diffe(u)}$, in the sense of $4\diffe(u)I_{e,u}(c)/c^2\to 1$ as $c/\sqrt{\diffe(u)}\to 0$, uniformly in $e,u$. That is, effective diffusivity $\diffe(u)$ yields a good approximation of the rate function $I_{e,u}(c)$ for $c\lesssim \sqrt{\diffe(u)}$ in two dimensions. However, Theorem~\ref{T.5.1} shows that this is not true for general flows in more dimensions. At this point we do not know how to explain this difference.

Motivated by a conjecture from \cite{ABP,Berrev} (see Corollary \ref{C.1.3} below), we are particularly interested in the strong flow asymptotics of the unique/minimal speed $\ce(Au,f)$ for
\beq \lb{1.8}
T_t + Au(x)\cdot\nabla T = \Delta T + f(T),
\eeq
with the {\it flow profile} $u$ as above and {\it flow amplitude} $A\in\bbR$ large. A natural question is which flow profiles are able to arbitrarily speed up fronts provided their amplitude is large enough (see \cite{ABP, Berrev, Heinze, KR, NR, RZ, VCKRR, ZlaPersym, ZlaPercol}). Having proved Theorem \ref{T.1.1},  this question in two dimensions becomes equivalent to the question of the asymptotics of $\diffe(Au)$. The latter is much simpler since $\diffe(Au)$ can be computed via \eqref{1.5} and \eqref{1.6} with $Au$ in place of $u$. In particular, it has been proved in \cite{BGW,FP,RZ} (see, e.g., Proposition 1.2 in \cite{RZ}) that in any dimension, $\limsup_{A\to\infty} \diffe(Au)<\infty$ when the equation
\beq \lb{1.9}
u\cdot\nabla\phi_e=u\cdot e
\eeq
has a solution $\phi_e\in H^1(\bbT^d)$, and $\lim_{A\to\infty} \diffe(Au)=\infty$ when \eqref{1.9} has no such solution. We therefore obtain the following characterization.

\begin{corollary} \lb{C.1.2}
Let $u(x)$ be a 1-periodic incompressible mean-zero $C^{1,\del}$ flow on $\bbR^2$, let $e\in\bbR^2$ be a unit vector and $f$ any combustion-type reaction.

(i) If \eqref{1.9} has a solution $\phi_e\in H^1(\bbT^2)$, then
\beq\label{1.10}
\limsup_{A\to \infty}\ce(Au,f)<\infty.
\eeq

(ii) If \eqref{1.9} has no solutions in $H^1(\bbT^2)$, then
\beq\label{1.11}
\lim_{A\to \infty}\ce(Au,f)=\infty.
\eeq
\end{corollary}

{\it Remark.} This result has been proved in \cite{RZ} for KPP reactions but is new for general $f$.
\smallskip

Of particular interest have recently been both {\it percolating} and {\it cellular flows}. 
Percolating flows possess streamlines (solutions of the ODE $X'=u(X)$) joining $x\cdot e =-\infty$ and $x\cdot e =\infty$, a special case being {\it shear flows} $u(x)=v(x_2,\dots,x_d)e_1$. Existence of such streamlines has obviously a strong effect on speed-up of fronts. Cellular flows, on the other hand, possess only closed streamlines, a prototypical example being $u_{\rm cell}(x)=\nabla^\perp (\sin 2\pi x_1 \sin 2\pi x_2)$ whose streamlines are depicted in Figure \ref{fig-cell}.
\begin{figure}[ht!]
 \centerline{\epsfxsize=0.36\hsize \epsfbox{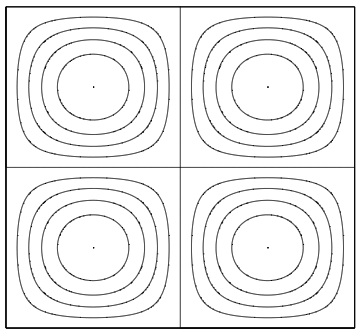}}
 \caption{Streamlines of the cellular flow $u_{\rm cell}$.}
 \label{fig-cell}
\end{figure}
Their effect on speed-up of fronts is therefore more subtle, with diffusion across a thin boundary layer near the flow separatrices playing an important role. The interest in these flows stems from them being ubiquitous in nature. They appear as a result of instabilities in fluids such as Rayleigh-B\' enard instability in heat convection, Taylor vortices in a Couette flow between rotating cylinders, or heat expansion driven Landau-Darrieus instability.

Corollary \ref{C.1.2} shows speed-up of fronts in the sense of \eqref{1.11} for both percolating and cellular flows but known estimates on the effective diffusivity and Theorem \ref{T.1.1} yield more precise asymptotics. For percolating flows 
\beq\label{1.12}
0<\liminf_{A\to\infty} \frac{\ce(Au,f)}A \le \limsup_{A\to\infty} \frac{\ce(Au,f)}A <\infty
\eeq
has been conjectured in \cite{ABP} and later proved for all $f$ in \cite{KR}, and can also be recovered from Theorem 1.2 in \cite{ZlaPercol} and Theorem \ref{T.1.1}. The asymptotic $\ce(Au_{\rm cell},f)\sim A^{1/4}$ for the above cellular flow has also been conjectured in \cite{ABP,Berrev} and obtained in \cite{NR} for KPP reactions, but the best result for general reactions has been $A^{1/5}\lesssim c^*_{e}(Au_{\rm cell},f)\lesssim A^{1/4}$ for $e=e_1$ \cite{KR}. 

Using Theorem \ref{T.1.1} and the estimate $\diffe(Au)\sim A^{1/2}$ from \cite{Kor}, we now obtain the conjectured asymptotic from \cite{ABP,Berrev} for general incompressible periodic cellular flows and all $f$. We also need to assume, as in \cite{Kor}, that the stream function of the flow has only non-degenerate critical points (otherwise the result is not true in general). That is, there is a periodic $C^{2,\delta}$ function $H:\bbR^2\to\bbR$ with only non-degenerate critical points such that $u=\nabla^\perp H \equiv (-H_{x_2},H_{x_1})$ and the complement of the the level set $H=0$ has only bounded connected components ({\it flow cells}). Notice that this allows for almost arbitrary periodic geometry of the flow cells. We will call such $u$ {\it periodic non-degenerate cellular flows}.

\begin{corollary} \lb{C.1.3}
Consider a 1-periodic non-degenrate cellular flow $u$ on $\bbR^2$, let $e\in\bbR^2$ be a unit vector, and $f$ any combustion-type reaction. Then
\beq\label{1.13}
0<\liminf_{A\to\infty} \frac{\ce(Au,f)}{A^{1/4}} \le \limsup_{A\to\infty} \frac{\ce(Au,f)}{A^{1/4}} <\infty.
\eeq
\end{corollary}




As mentioned above, we also address the closely related question of {\it quenching} of reaction by cellular flows. Quenching occurs in the Cauchy problem for \eqref{1.1} with initial data $T(0,x)=T_0(x)$ when $\|T(t,\cdot)\|_\infty\to 0$ as $t\to\infty$. This happens for ignition reactions when $T_0$ is small in some sense so that $T(\tau,\cdot)$ becomes uniformly smaller than the ignition temperature for some $\tau>0$ (and thus $T$ solves \eqref{1.3} for $t>\tau$ by the maximum principle). But sometimes even large initial data can be quenched with the help of enhanced diffusion due to mixing by strong flows. 

It has been proved in \cite{FKR} that the flow $Au_{\rm cell}$ quenches initial data supported in a strip of width $L$ provided the amplitude $A\gtrsim L^4\ln L$. The authors also conjectured that the factor $\ln L$ can be removed (heuristically, the minimal quenching amplitude should be such that $\ce(Au,f)\sim L$, i.e., $A\sim L^4$). 
We prove this conjecture for all periodic non-degenerate cellular flows on $\bbR^2$, which are also symmetric across the $x_2$ axis (i.e., the stream function $H$ is odd in $x_1$). We will call such periodic non-degenerate flows, which include $u_{\rm cell}$, {\it symmetric}. 

\begin{theorem} \lb{T.1.4}
Consider a $1$-periodic symmetric non-degenrate cellular flow $u$ on $\bbR^2$. There are $\gamma_\tht>0$ (independent of $u$) and  $C_{u,\tht}>0$ such that if $f$ is an ignition reaction with ignition temperature $\tht>0$ and $\|f(s)/s\|_\infty\le \gamma_\tht$, then initial data $T_0(x)\in[0,1]$ supported in $[-L,L]\times\bbR$ are quenched whenever $A\ge C_{u,\tht} L^4$.
\end{theorem}

{\it Remarks.} 1. We note that the bound $\gamma_\tht$ on $f$ is in fact necessary. It is easy to show that if $f$ is large enough, then even a single cell with Dirichlet boundary conditions can support the reaction for any $A$ \cite{FKR}.
\smallskip

2. As in \cite{FKR}, scaling shows that $\|f(s)/s\|_\infty\le \gamma_\tht$ can be replaced by the requirement that the period of $u$ be smaller than $\|f(s)/s\|_\infty^{1/2} \gamma_\tht^{-1/2}$.
\smallskip

The rest of the paper is organized as follows. In Section \ref{S2} we introduce our main technical tool, Theorem \ref{T.2.1}, relating short and long time diffusivity of \eqref{1.3} in any dimension. In Section \ref{S3} we prove Theorem \ref{T.3.1}, a relation between the speed and the width of a front in two dimensions, and then Theorem \ref{T.1.1}. In Section \ref{S4} we prove Theorem \ref{T.1.4} and in Section \ref{S5} we provide a counterexample to Theorem \ref{T.1.1} in three and more dimensions.

It is a pleasure to thank James Nolen for useful discussions. The author has been supported in part by the NSF grant DMS-0901363 
and by an Alfred P.~Sloan Research Fellowship.

\section{Short Time Diffusivity of Periodic Flows} \lb{S2}

In this section we show that there is a close relation between short- and long-time diffusivity of the parabolic operator in  \eqref{1.3}. The results contained here are valid in any dimension.

Consider the stochastic process $X_t^{x}$ starting at $x\in \bbR^d$ and satisfying the stochastic differential equation
\beq \lb{2.1}
dX^{x}_t=\sqrt{2}\,dB_t - u(X^{x}_t)dt, \qquad X^{x}_0=x,
\eeq
where $B_t$ is a normalized Brownian motion on $\bbR^d$ with $B_0=0$ (defined on a probability space $(\Omega,\calB_\infty,\bbP_\Omega)$). By Lemma 7.8 in \cite{Oks}, we have that if $\psi$ solves \eqref{1.3} with $\psi(0,x)=\psi_0(x)$,
then
\beq \lb{2.3}
\psi(t,x) = \int_{\bbR^d} \psi_0(y)k_t(x,y) \,dy = \bbE_{\Omega} \big( \psi_0(X^{x}_t) \big),
\eeq
with $k_t(x,y)$ the fundamental solution for \eqref{1.3} and $\bbE_\Omega$ expectation with respect to $\omega\in\Omega$. That is, $k_t(x,.)$ is the density for $X_t^x$. 

We also have that if $T$ solves \eqref{1.1} with $T(0,x)=T_0(x)$ and $\psi_0(x)=T_0(x)\ge 0$, then by the comparison principle \cite{Sm}, for all $t,x$,
\beq \lb{2.4}
0\le \psi(t,x)\le  T(t,x)\le e^{t\|f(s)/s\|_\infty} \psi(t,x).
\eeq
We will use this relation to obtain short time upper estimates on the solution of \eqref{1.1}.

We start the study of short time diffusivity for \eqref{1.3} with noting that uniformly in $x\in\bbR^d$, 
\beq \lb{2.5}
\lim_{t\to\infty} \bbE_\Om \left(\frac {\left|(X_t^x-x)\cdot e\right|^2}{2t} \right) = \diffe(u).
\eeq
This  is based on the fact that
\[
 \bbE_\Om \left(\left|(X_t^x-x)\cdot e\right|^2 \right) = \int_{\bbR^d} |(y-x)\cdot e|^2 k_t(x,y)\,dy =  \int_{\bbR^d} |(\sqrt t\,y-x)\cdot e|^2 t^{d/2} k_t(x,\sqrt t\,y)\,dy
\]
and the following estimates on $k_t(x,y)$:
\begin{gather*}
t^{d/2} k_t(x,\sqrt{t}\,y)  \to k_1^{*}(0,y)  \quad \text{in $L^2(\bbR^d)$ as $t\to\infty$, uniformly in $x$;}\\
0\le k_t(x,y)  \le Ct^{-d/2} e^{-|x-y|^2/Ct} \quad \text{for some $u$-dependent $C>0$}.
\end{gather*}
The first is the standard homogenization limit (see, e.g., \cite{JKO, KM}) with 
\[
k_t^{*}(x,y)\equiv \frac 1{\sqrt{{\rm det}(\diff(u))}\, (4\pi t)^{d/2}} e^{-(y-x)\cdot \diff(u)^{-1}(y-x)/4t}
\]
 the fundamental solution of \eqref{1.4}, the second is the Nash-Aronson estimate for mean-zero flows (see, e.g., \cite{Norris}). The limit \eqref{2.5} now follows from
 \[
 \int_{\bbR^d} |y\cdot e|^2 k_1^{*}(0,y)\,dy = \int_{\bbR^d} \frac{|\sqrt{\diff(u)}\,z\cdot e|^2}{(4\pi)^{d/2}} e^{-z^2/4} dz = 
\int_{\bbR^d} \frac{|z\cdot \sqrt{\diff(u)}\,e|^2}{(4\pi)^{d/2}} e^{-z^2/4} dz  = 2\diffe(u).
 \]
 
The main result of this section is a lower bound on short time diffusivity for \eqref{1.3}:

\begin{theorem} \lb{T.2.1} 
There is $C>0$ such that  for any $\tau\ge 1$, any 1-periodic incompresible mean-zero Lipschitz flow $u$ and any  $\alpha>0$ there are $x\in\bbR^d$ and $t\in[0,\tau]$ such that 
\beq \lb{2.6}
\bbP_\Om \left(\left|(X_t^x-x)\cdot e\right| \ge \alpha \sqrt{\tau\diffe(u)} \right) \ge 1-C\alpha.
\eeq
%
%
\end{theorem}


\begin{proof}
We first note that if $\tilOm\equiv \bbT^d\times\Om$ is equipped with the product probability measure, then \eqref{2.5} gives
\beq \lb{2.7}
\lim_{t\to\infty} \bbE_\tilOm \left(\frac {\left|(X_t^x-x)\cdot e\right|^2}{2t} \right) = \diffe(u),
\eeq
where the expectation is with respect to $(x,\omega)\in\tilOm$. Then we have


\begin{lemma} \lb{L.2.2}
There is $\til C>0$ such that for any $\tau\ge 1$ and $u$ as in Theorem \ref{T.2.1},
\beq \lb{2.8}
\bbE_\tilOm \left(\left| (X_\tau^x -x)\cdot e \right|^2 \right) \ge \til C\tau \diffe(u),
\eeq
\end{lemma}

\begin{proof}
Let us first assume $\tau=1$. For $x\in\bbT^d$, let $\til X_{t}^x\equiv X_{t}^x \, {\rm mod}\, 1\in\bbT^d$ be the process corresponding to \eqref{1.3} on $\bbT^d$. We note that $\til X_t^x$ is uniformly distributed over $\bbT^d$ as a random variable on $\tilOm$. Indeed, if $B\subseteq\bbT^d$ and $\psi_0(x)=\chi_B(x)$, then for each $t\ge 0$, 
\[
\bbP_\tilOm (\til X_t^x\in B) = \int_{\bbT^d} \psi(t,x) \,dx = \int_{\bbT^d} \psi(0,x) \,dx = |B|
\]
because the evolution \eqref{1.3} preserves the total mass of $\psi$.

We next let $Y_t^x\equiv X_t^x-X_{t-1}^x$ and $Z_t^x\equiv Y_t^x\cdot e$. Then periodicity of $u$ implies
\[
Y_t^x = X_1^{X_{t-1}^x}-X_{t-1}^x = X_1^{\til X_{t-1}^x}-\til X_{t-1}^x
\]
in law. Since the $\til X_{t-1}^x$ for all $t\ge 1$ are identically distributed as random variables on $\til\Omega$, the same is true for the increment displacements $Y_t^x$ as well as the $Z_t^x$.  In particular, for each $t\ge 1$, 
\beq \lb{2.9}
\bbE_\tilOm(|Z_t^x|^2) = \bbE_\tilOm(|Z_1^x|^2).
\eeq

We also have $(X_N^x-x)\cdot e=\sum_{n=1}^N Z_n^x$ for any $N\in\bbN$  and so by \eqref{2.7},
\beq \lb{2.10}
\sum_{n,m=1}^N \bbE_\tilOm \left( Z_n^x Z_m^x \right) = 2N\diffe(u) + o(N).
\eeq
Since $\left| \bbE_\tilOm \left( Z_n^x Z_m^x \right) \right| \le \bbE_\tilOm(|Z_1^x|^2)$ is obvious from \eqref{2.9} and the Schwarz inequality, we will obtain \eqref{2.8} for $\tau=1$ if we can show the existence of $u$-independent $M\in\bbN$ and $\gamma>0$ such that
\beq \lb{2.11}
\left| \bbE_\tilOm \left( Z_n^x Z_m^x \right) \right| \le 2^{-\gamma|m-n|} \bbE_\tilOm(|Z_1^x|^2).
\eeq
whenever $|m-n|\ge M+1$.

We denote $h_t(x,y)\equiv \sum_{j\in\bbZ^d} k_t(x,y+j)$ the fundamental solution for \eqref{1.3} on $\bbT^d$. We then have 
\beq \lb{2.12}
\int_{\bbT^d} h_t(x,y)\,dy = \int_{\bbT^d} h_t(x,y)\,dx = 1.
\eeq
By Lemma 5.6 in \cite{CKRZ}, there is $M\in\bbN$ such that for all $u$ as above, $\|h_M(.,.)-1\|_\infty \le \tfrac 12$. The maximum principle gives $\|h_t(.,.)-1\|_\infty \le \|h_s(.,.)-1\|_\infty$ for $t\ge s$ and this together with $(h_t-1)*(h_s-1)=h_{t+s}-1$ (from \eqref{2.12}) implies for all $m\ge M$ and $\gamma\equiv (2M)^{-1}$,
\beq \lb{2.13}
\|h_m(.,.)-1\|_\infty \le 2^{-\gamma (m+1)}.
\eeq
As a final prerequisite, we note that 
\beq \lb{2.14}
\int_{(\bbT^d)^2} \sum_{j\in\bbZ^d} (y+j-x) k_t(x,y+j)\,dxdy=0.
\eeq
This can be obtained by taking the initial datum $\psi_0(x)=\chi_{\bbT^d}(x)$ in \eqref{1.3} on $\bbR^d$ and evaluating
\[
\frac d{dt} \int_{\bbR^d} x\psi\, dx = \int_{\bbR^d} x\Delta\psi - x \nabla\cdot(u\psi)\, dx = \int_{\bbR^d} u\psi\, dx = \int_{\bbT^d} u(x) \sum_{j\in\bbZ^d} \psi(t,x+j)\, dx = 0,
\]
where we used integration by parts, the fact that $\sum_{j\in\bbZ^d} \psi(t,x+j)\equiv 1$ and $u$ being mean-zero. Thus for each $t\ge 0$ (recall that $\bbT^d=[-\tfrac 12,\tfrac 12]^d$),
\[
0 = \int_{\bbR^d} x\psi(t,x)\, dx =  \int_{\bbR^d\times\bbT^d} x k_t(x,y)\,dxdy = \int_{\bbR^d\times\bbT^d} (x-y) k_t(x,y)\,dxdy
\]
because $\int_{\bbR^d} yk_t(x,y)\,dx = y$. Then by periodicity,
\[
0= \int_{(\bbT^d)^2} \sum_{j\in\bbZ^d} (x+j-y) k_t(x+j,y)\,dxdy = -\int_{(\bbT^d)^2} \sum_{j\in\bbZ^d} (y-j-x) k_t(x,y-j)\,dxdy.
\]

In what follows we denote $x_e\equiv x\cdot e$ for $x\in\bbR^d$. For $m-n\ge M+1$ we have by $k_t*k_s=k_{t+s}$,
\begin{align*}
\left| \bbE_\tilOm \left( Z_n^x Z_m^x \right) \right| =  \bigg| \int_{(\bbR^d)^5}  & (x_n-x_{n-1})_e(x_m-x_{m-1})_e k_{n-1}(x,x_{n-1})k_1(x_{n-1},x_n) \\
& k_{m-1-n}(x_n,x_{m-1})k_1(x_{m-1},x_m)\,dxdx_{n-1}dx_ndx_{m-1}dx_m \bigg| \\
=  \bigg| \int_{(\bbT^d)^5} & \sum_{j,l\in\bbZ^d}  (x_n+j-x_{n-1})_e(x_m+l-x_{m-1})_e h_{n-1}(x,x_{n-1}) k_1(x_{n-1},x_n+j) \\
& h_{m-1-n}(x_n,x_{m-1})k_1(x_{m-1},x_m+l)\,dxdx_{n-1}dx_ndx_{m-1}dx_m \bigg|.
\end{align*}
Here we have used the fact that 
\begin{align*}
\sum_{p,q\in\bbZ^d}  & (x_m+p-x_{m-1}-q)_e  k_{m-1-n}(x_n,x_{m-1}+q)k_1(x_{m-1}+q,x_m+p) \\
= & \sum_{l\in\bbZ^d} (x_m+l-x_{m-1})_e \sum_{q\in\bbZ^d} k_{m-1-n}(x_n,x_{m-1}+q) k_1(x_{m-1}+q,x_m+l+q) \\
= & \sum_{l\in\bbZ^d} (x_m+l-x_{m-1})_e k_1(x_{m-1},x_m+l) \sum_{q\in\bbZ^d} k_{m-1-n}(x_n,x_{m-1}+q) \\
= & \sum_{l\in\bbZ^d} (x_m+l-x_{m-1})_e h_{m-1-n}(x_n,x_{m-1}) k_1(x_{m-1},x_m+l)
\end{align*}
(due to periodicity of $u$) and similarly 
\begin{align*}
\sum_{p,q\in\bbZ^d}  & (x_n+p-x_{n-1}-q)_e  k_{n-1}(x,x_{n-1}+q)k_1(x_{n-1}+q,x_n+p) \\
= & \sum_{j\in\bbZ^d} (x_n+j-x_{n-1})_e h_{n-1}(x,x_{n-1}) k_1(x_{n-1},x_n+j).
\end{align*}
The integral with respect to $x$ can now be eliminated along with $h_{n-1}(x,x_{n-1})$ because  $\int_{\bbT^d} h_{n-1}(x,x_{n-1})\,dx=1$. Notice that \eqref{2.14} gives
\[
\int_{(\bbT^d)^2} \sum_{l\in\bbZ^d} (x_m+l-x_{m-1})_e k_1(x_{m-1},x_m+l)\,dx_{m-1}dx_m = 0,
\]
and so \eqref{2.13} and Schwarz inequality imply
\begin{align*}
2^{\gamma(m-n)} & \left| \bbE_\tilOm  \left( Z_n^x Z_m^x \right) \right| 
=  2^{\gamma(m-n)} \bigg| \int_{(\bbT^d)^4} \sum_{j,l\in\bbZ^d} (x_n+j-x_{n-1})_e(x_m+l-x_{m-1})_e k_1(x_{n-1},x_n+j) \\
& \phantom{mmmmmmmmmmmmm} [h_{m-1-n}(x_n,x_{m-1})-1]k_1(x_{m-1},x_m+l)\,dx_{n-1}dx_ndx_{m-1}dx_m \bigg| \\
\le & \bigg| \int_{(\bbT^d)^4} \sum_{j,l\in\bbZ^d} |(x_n+j-x_{n-1})_e|^2 k_1(x_{n-1},x_n+j) 
k_1(x_{m-1},x_m+l)\,dx_{n-1}dx_ndx_{m-1}dx_m \bigg|^{1/2}  \\
& \bigg| \int_{(\bbT^d)^4} \sum_{j,l\in\bbZ^d} |(x_m+l-x_{m-1})_e|^2 k_1(x_{n-1},x_n+j) 
k_1(x_{m-1},x_m+l)\,dx_{n-1}dx_ndx_{m-1}dx_m \bigg|^{1/2} \\
= & \bigg| \int_{(\bbT^d)^2} \sum_{j\in\bbZ^d} |(x_n+j-x_{n-1})_e|^2 k_1(x_{n-1},x_n+j) \,dx_{n-1}dx_n \bigg|^{1/2} \\
 & \bigg| \int_{(\bbT^d)^2} \sum_{l\in\bbZ^d} |(x_m+l-x_{m-1})_e|^2 k_1(x_{m-1},x_m+l) \,dx_{m-1}dx_m \bigg|^{1/2} \\
 = & \bbE_\tilOm  \left( |Z_n^x|^2\right)^{1/2} \bbE_\tilOm  \left( |Z_m^x|^2\right)^{1/2}. 
\end{align*}
Now \eqref{2.9} yields \eqref{2.11}, finishing the proof for $\tau=1$. The general case is identical, this time with $Y_n^x$ being $X_{n\tau}^x-X_{(n-1)\tau}^x$, the same $\gamma$, $M$ and $\til C$, and $2N$ replaced by $2N\tau$ in \eqref{2.10} (one actually gets $\til C\to 2$ as $\tau\to\infty$).
\end{proof}


We will now prove \eqref{2.6} with $C\equiv 10\til C^{-1/2}$. Assume, towards contradiction, that for some $u$ there are $\tau\ge 1$ and $\alpha>0$ such that for any $x\in\bbR^d$ and any $t\in[0,\tau]$,
\beq \lb{2.15}
\bbP_\Om \left(\left|(X_t^x-x)\cdot e\right| < \alpha \sqrt{\tau\diffe(u)} \right) > C\alpha.
\eeq

We first claim that for each $x\in\bbR^d$,
\beq \lb{2.16}
\bbP_\Om \left( \forall t\in[0,\tau] \,\bigg|\, \left|(X_t^x-x)\cdot e\right| < \frac{4\sqrt{\tau\diffe(u)}} C  \right) \ge \frac 12.
\eeq
Indeed, if this is not true, let $x$ be such that the probability in \eqref{2.16} is less than $\tfrac 12$. This means that there is a subset $\Omega'\subseteq\Omega$ of measure more than $\tfrac 12$ such that if $t_j(\omega)\ge 0$ is the first time the (almost surely continuous in $t$) path $X_t^x=X_t^x(\omega)$  hits the set 
\[
H_{j}\equiv \big\{ y\in\bbR^d \,\big|\, |(y-x)\cdot e|=2j\alpha\sqrt{\tau\diffe(u)} \big\},
\]
then $t_j(\omega)\le\tau$ for each $\omega\in\Omega'$ and $j=0,\dots,\lfloor \tfrac 2{C\alpha} \rfloor$.
Now the strong Markov property of the process $X_t^x$, the fact that $\tau-t_j\in[0,\tau]$, and \eqref{2.15} imply that for $j=0,\dots,\lfloor \tfrac 2{C\alpha} \rfloor$ the following conditional probability satisfies
\[
\bbP_\Om \left(\left|(X_\tau^x-x)\cdot e\right| \in \left( (2j-1) \alpha \sqrt{\tau\diffe(u)},(2j+1) \alpha \sqrt{\tau\diffe(u)} \right) \,\Big|\, \calB_{t_j} \right) > C\alpha \chi_{\Om'}(\omega),
\]
with $\calB_{t_j}$ the $\sigma$-algebra corresponding to the stopping time $t_j$.
Thus
\[
\bbP_\Om \left(\left|(X_\tau^x-x)\cdot e\right| \in \left( (2j-1) \alpha \sqrt{\tau\diffe(u)},(2j+1) \alpha \sqrt{\tau\diffe(u)} \right) \right) > \frac{C\alpha}2
\]
for $j=0,\dots,\lfloor \tfrac 2{C\alpha} \rfloor$.
Since $\frac{C\alpha}2 (\lfloor \tfrac 2{C\alpha} \rfloor+1)>1$, this is a contradiction, thus proving \eqref{2.16}.

Now \eqref{2.16} and the almost sure continuity of $X_t^x$ in $t$ show for each $x\in\bbR^d$ and each $j\ge 1$,
\[
\bbP_\Om \left( \left|(X_\tau^x-x)\cdot e\right| \ge \frac{4j\sqrt{\tau\diffe(u)}} C  \right) \le \left( \frac 12\right)^j.
\]
That, however, means
\[
\bbE_\tilOm \left(\left| (X_\tau^x -x)\cdot e \right|^2 \right) \le \frac{16\tau\diffe(u)}{C^2} \sum_{j\ge 1} j^2 \left[ \left(\frac 12\right)^{j-1} - \left(\frac 12\right)^{j} \right] = \frac{96\tau \diffe(u)}{C^2} <  \til C\tau \diffe(u),
\]
contradicting \eqref{2.8}. This finishes the proof of Theorem \ref{T.2.1}.
\end{proof}

\section{Pulsating Front Speed in 2D} \lb{S3}

We will now prove Theorem \ref{T.1.1}. The upper bound in \eqref{1.7} is immediate from the results in \cite{RZ}. Indeed, the same bound has been proved there for KPP reactions with $C_2(f)=C\sqrt{f'(0)}(1+\sqrt{f'(0)})$. One thus only needs to apply this result to some KPP $\til f$ such that $f\le\til f$ and $\|f(s)/s\|_\infty=\til f'(0)$, and use the fact that $\ce(u,f)\le \ce(u,\til f)$ due to $f\le\til f$. 

We will now prove the lower bound in \eqref{1.7} using Theorem \ref{T.2.1} coupled with the following bound on the width of the front in terms of its speed. We note that we can assume $T_t>0$. This has been proved in \cite{BH} for all ignition reactions (and also for positive reactions with $f'(0)>0$), and the lower bound in \eqref{1.7} for any ignition reaction $\til f \le f$ proves the lower bound for $f$ because again $\ce(u,\til f)\le \ce(u,f)$.

\begin{theorem} \lb{T.3.1}
There is $C_0>0$ such that for any 1-periodic incompressible mean-zero $C^{1,\del}$ flow $u$ on $\bbR^2$, any unit vector $e\in\bbR^2$, and any combustion-type reaction $f$ the following holds. If $f\ge m\chi_{[\zeta,\xi]}$ for some $m>0$ and $0<\zeta<\xi<1$, $\eps\in(0,(\xi-\zeta)/2)$, and $T(t,x)$ is a pulsating front for \eqref{1.1} with speed $c>0$ and $T_t>0$, then  there is $z\in\bbR$ such that any connected set $B$ with
\[
B\subseteq \{x\in\bbR^2 \,|\, x\cdot e\ge z+ct+cC_0(m^{-1}+\eps^{-2})+2 \text{ and } T(t,x)\ge \zeta+\eps  \} \equiv B_t^+
\]
or
\[
B\subseteq \{x\in\bbR^2 \,|\, x\cdot e\le z+ct-cC_0(m^{-1}+\eps^{-2})-2 \text{ and } T(t,x)\le \xi-\eps  \} \equiv B_t^-
\] 
satisfies $\diam(B)\le \tfrac 1{10}$.
\end{theorem}

{\it Remark.} The above form of this result will be sufficient for our purposes. Its proof in fact shows that the set of $x$ such that $T(t,x)\ge \zeta+\eps$ resp. $T(t,x)\le \xi-\eps$ covers at most 1\% of any unit square lying in the halfplane $x\cdot e\ge z+ct+cC_0(m^{-1}+\eps^{-2})+2$ resp. $x\cdot e\le z+ct-cC_0(m^{-1}+\eps^{-2})-2$ (and this bound decreases as $C_0$ increases). That is, for any $t$, outside of a strip of width $2cC_0(m^{-1}+\eps^{-2})+4$, values of $T(t,\cdot)$ are mostly outside of $[\zeta+\eps,\xi-\eps]$. This yields a bound on the width of the front in terms of its speed.

\begin{proof}
Let $T(t,x)=U(x\cdot e-ct,x)$ with $U$ satisfying \eqref{1.2}, so that 
\beq \lb{3.1}
-cU_s+u\cdot \nabla_x U + u\cdot e U_s = \Delta_x U+U_{ss}+2e\cdot\nabla_xU_s+f(U).
\eeq
Integrating this over $\Gamma\equiv\bbR\times\Gamma_0\equiv\bbR\times[0,1]^2$ and using 1-periodicity of $U$ in $x$, \eqref{1.2}, and $u$ being incompressible and mean zero, we get 
\beq \lb{3.2}
\int_\Gamma f(U(s,x))dsdx = c.
\eeq
Similarly, multiplying \eqref{3.1} by $U$ and integrating over $\Gamma$  yields 
\[
\frac c2+\int_\Gamma |\nabla_x U +e U_s|^2 \,dsdx= \int_\Gamma f(U(s,x)) U(s,x)\,dsdx.
\]
The right hand side is bounded above by $c$ thanks to \eqref{3.2} and so
\begin{gather} 
\int_\Gamma f(T(t,x)) \,dtdx = 1, \lb{3.3}
\\ \int_\Gamma |\nabla_x T(t,x)|^2 \,dtdx \le \frac 12 \lb{3.4}.
\end{gather}

Let $\bar T(t)\equiv\int_{\Gamma_0} T(t,x)\,dx\in[0,1]$ so that $\bar T_t>0$ because $T_t>0$. Assume that $\bar T(t)\in[\zeta+\tfrac \eps 2, \xi+\tfrac \eps 2]$ for some $t$. Then the Poincar\' e inequality $\|T(t,\cdot)-\bar T(t)\|_{L^2(\Gamma_0)}\le C_1\|\nabla_xT(t,\cdot)\|_{L^2(\Gamma_0)}$ for some $C_1\ge 2$ gives either $\|T(t,\cdot)-\bar T(t)\|_{L^2(\Gamma_0)}\le\tfrac \eps 4$ or $\|\nabla_xT(t,\cdot)\|_{L^2(\Gamma_0)}\ge \tfrac \eps{4C_1}$. In the first case the set of $x\in\Gamma_0$ with $|T(t,x)-\bar T(t)|\ge \tfrac \eps 2$ has measure less than $\tfrac 12$ and so $\int_{\Gamma_0} f(T(t,x))dx\ge \tfrac m2$. Thus \eqref{3.3} and \eqref{3.4} show that the interval of all $t$ with $\bar T(t)\in[\zeta+\tfrac \eps 2, \xi+\tfrac \eps 2]$ has length at most $\tfrac 2m+\tfrac {8C_1^2}{\eps^2}$.

Assume now $\bar T(t)\le\zeta+\tfrac \eps 2$ and that there is a connected set $B\subseteq \{x\in\Gamma_0 \,|\,  T(t,x)\ge \zeta+\eps  \}$ with $\diam(B)\ge \tfrac 1{10}$ (such $t$ form an interval due to $T_t>0$). Projection of $B$ on one of the axes (let us say $x_1$) is then an interval $I$ of length more than $\tfrac 1{16}$. If for each $x_1\in I$ there is $x_2\in[0,1]$ such that $T(t,x_1,x_2)\le\zeta+\tfrac {3\eps}4$, then $\|\nabla_xT(t,\cdot)\|_{L^2(\Gamma_0)}\ge \tfrac \eps{16}$. If on the other hand $T(t,y,x_2)\ge\zeta+\tfrac {3\eps}4$ for some $y\in I$ and all $x_2\in[0,1]$, then there is $J\subseteq[0,1]$ of measure at least $\tfrac 12$ such that either $T(t,x_1,x_2)\ge \tfrac{5\eps}8$ for all $(x_1,x_2)\in [0,1]\times J$ or $T(t,x_1,x_2)\le \tfrac{5\eps}8$ for each $x_2\in J$ and some $x_2$-dependent $x_1\in[0,1]$. In the first case  $\|T(t,\cdot)-\bar T(t)\|_{L^2(\Gamma_0)}\ge \tfrac \eps 8$, and the Poincar\' e inequality gives $\|\nabla_xT(t,\cdot)\|_{L^2(\Gamma_0)}\ge \tfrac \eps{8C_1}$. In the second case $\|\nabla_xT(t,\cdot)\|_{L^2(\Gamma_0)}\ge \tfrac \eps{12}$. In either case we have $\|\nabla_xT(t,\cdot)\|_{L^2(\Gamma_0)}\ge \tfrac \eps{8C_1}$, so the interval of the $t$ above has length at most $\tfrac {32C_1^2}{\eps^2}$.

The same is true for $t$ such that $\bar T(t)\ge\xi-\tfrac \eps 2$ and there is a connected set $B\subseteq \{x\in\Gamma_0 \,|\,  T(t,x)\le \xi-\eps  \}$ with $\diam(B)\ge \tfrac 1{10}$. Thus there is an interval $[a,b]$ with $b-a\le \tfrac 2m+\tfrac {72C_1^2}{\eps^2}$ such that for $t\le a$ connected subsets of  $\{x\in\Gamma_0 \,|\,  T(t,x)\ge \zeta+\eps  \}$ have diameter at most $\tfrac 1{10}$ and the same is true for $t\ge b$ and connected subsets of  $\{x\in\Gamma_0 \,|\,  T(t,x)\le \xi-\eps  \}$. 

Finally, let $z\equiv -c\tfrac{a+b}2$ and $C_0\equiv 36C_1^2$, and assume that for some $t$ a set $B\subseteq B_t^+$ contains a point $\til x$. Then with $\lfloor x\rfloor$ the integer part of $x$,
\[
T(t,x)=U(x\cdot e-ct,x)=U(x\cdot e-ct,x-\lfloor\til x\rfloor)=T(t-\tfrac{\lfloor\til x\rfloor \cdot e}c, x-\lfloor\til x\rfloor).
\]
We have using $\til x\in B$,
\[
t-\frac{\lfloor\til x\rfloor \cdot e}c \le \frac {\sqrt 2}c- \frac zc -C_0(m^{-1}+\eps^{-2})-\frac 2c \le \frac{a+b}2 - 36 C_1^2 (m^{-1}+\eps^{-2}) \le a 
\]
and so $\diam(B)\le\tfrac 1{10}$  (in fact, to handle $B$ not lying entirely inside a square with integer corners, we need to consider $\Gamma_0=[-\tfrac 1{10}, 1+ \tfrac 1{10}]^2$, which only changes $C_0$ by a fixed factor). 
A similar argument takes care of sets $B\subseteq B_t^-$, thus finishing the proof. 

Notice that if $B$ is not required to be connected in the proof, then one still obtains the bound $\tfrac 1{16}$ on the size of its projection on the axes. This proves the remark.
\end{proof}

\begin{proof}[Proof of Theorem \ref{T.1.1}]
As mentioned above, we only need to prove the lower bound in \eqref{1.7} and only assuming $T_t>0$.
For $\zeta\in(0,1)$ let $\xi\equiv 1-\tfrac{(1-\zeta)^2}8$ and $\eps\equiv \tfrac{(1-\zeta)^2}8$, and  pick $M_\zeta,\alpha_\zeta>0$ so that with $C$ from Theorem \ref{T.2.1},
\beq \lb{3.5}
\frac{\zeta+\eps}{\xi-\eps} + \frac{1-(\xi-\eps)e^{-M_\zeta}}{1-(\zeta+\eps)} < 1-C\alpha_\zeta.
\eeq
This is possible because if  $\zeta'>\zeta+\eps$ and $1-\tfrac{(1-\zeta')^2}2<\xi-\eps$ (e.g., $\zeta'\equiv \tfrac{1+7\zeta}8$), then
\[
\frac{\zeta+\eps}{\xi-\eps} + \frac{1-(\xi-\eps)}{1-(\zeta+\eps)} < \frac{\zeta'}{1-\tfrac{(1-\zeta')^2}2}+ \frac{\tfrac{(1-\zeta')^2}2}{1-\zeta'} = 1-\frac{(1+\zeta'^2)(1-\zeta')}{2(1+2\zeta'-\zeta'^2)} < 1.
\]
Let us choose $\zeta$ so that $f$ is strictly positive on $[\zeta,\xi]$ and denote $m\equiv\min_{s\in[\zeta,\xi]} f(s)>0$. It is sufficient to assume that  
\beq \lb{3.6}
m\le M_\zeta \zeta \qquad\text{and}\qquad f(s)\le \frac m\zeta s \quad\text{for $s\in[0,1]$.}
\eeq
Indeed, otherwise consider $\til f(s) \equiv \min \{f(s),\tfrac m\zeta s,M_\zeta s\}$ (which satisfies \eqref{3.6} with $\til m\equiv \min_{s\in[\zeta,\xi]} \til f(s) = \min\{m,M_\zeta \zeta\}$ in place of $m$) instead of $f$ and then use $\ce(u,\til f)\le \ce(u, f)$. Note also that \eqref{1.7b} for all $m_\zeta(f)\in (0,M_\zeta\zeta)$ proves \eqref{1.7b} for all $m_\zeta(f)>0$ (with a different $C_\zeta$). So let us assume \eqref{3.6}.

Let $T$ be a pulsating front for \eqref{1.1} in direction $e$ and with speed $\ce(u,f)$, assume $z=0$ in Theorem \ref{T.3.1} (otherwise shift $T$ in time by $z/\ce(u,f)$), and let  $\tau\equiv M_\zeta\zeta/m\ge 1$.  Theorem \ref{T.2.1} and \eqref{3.5} show that there is $x\in\bbR^2$ and $t\in[0,\tau]$ such that either
\beq \lb{3.7}
\bbP_\Om \left((X_t^x-x)\cdot e \le -\alpha_\zeta \sqrt{\tau\diffe(u)} \right) > \frac{\zeta+\eps}{\xi-\eps}
\eeq
or
\beq \lb{3.8}
\bbP_\Om \left((X_t^x-x)\cdot e \ge \alpha_\zeta \sqrt{\tau\diffe(u)} \right) >  \frac{1-(\xi-\eps)e^{-t m/\zeta}}{1-(\zeta+\eps)}.
\eeq
We now claim that
\beq \lb{3.12}
\alpha_\zeta  \sqrt{\tau\diffe(u)}- \tfrac 1{10} \le \ce(u,f)t + 2\ce(u,f)C_0(m^{-1}+\eps^{-2})+8.
\eeq
If so, then $t\le\tau=M_\zeta\zeta/m$ gives
\[
\frac {C_1 \sqrt{m\diffe(u)}-C_2m}{C_3+C_4m}\le  \ce(u,f)
\]
with some $\zeta$-dependent positive constants. Since $\diffe(u)\ge 1$ by \eqref{1.6}, this then gives $C_5\sqrt{m\diffe(u)}\le \ce(u,f)$ for some $C_5>0$ and all small enough $m>0$, thus proving \eqref{1.7} and \eqref{1.7b} for all $m_\zeta(f)>0$.

It remains to prove \eqref{3.12}, and it is sufficient to consider $\alpha_\zeta  \sqrt{\tau\diffe(u)}\ge \tfrac 1{10}$. Assume first \eqref{3.7}. Due to spatial periodicity of $u$ we can assume for $x,t$ from \eqref{3.7},
\beq \lb{3.9}
\left| x\cdot e -\ce(u,f)t -\ce(u,f)C_0(m^{-1}+\eps^{-2})-4 \right|\le 1.
\eeq
We can also assume
\beq \lb{3.10}
T(t,x)\le \zeta+\eps
\eeq
at the expense of changing \eqref{3.7} to
\beq \lb{3.11}
\bbP_\Om \left((X_t^x-x)\cdot e \le \tfrac 1{10}-\alpha_\zeta \sqrt{\tau\diffe(u)} \right) > \frac{\zeta+\eps}{\xi-\eps}.
\eeq
This is because Theorem \ref{T.3.1} shows that the largest connected set $B$ of points satysfying \eqref{3.9} but not \eqref{3.10} and containing $x$ (if it is non-empty) has $\diam(B)< \tfrac 1{10}$  (recall that $z=0$). Hence almost sure continuity of $X_t^x$ in $t$ shows for any $\beta\le -\tfrac 1{10}$, 
\[
\bbP_\Om \left((X_t^x-x)\cdot e \le \beta \right) \le \sup_{\substack{y\in \partial B\\ s\in[0,t]}} \bbP_\Om \left((X_s^y-x)\cdot e \le \beta \right) \le \sup_{\substack{y\in \partial B\\ s\in[0,t]}} \bbP_\Om \left((X_s^y-y)\cdot e \le \beta +\tfrac 1{10} \right).
\]

Let $\Gamma$ be the union of all connected components of the set $\Gamma'\equiv \{y\in\bbR^2\,|\, T(0,y)\le \xi-\eps  \}$ lying entirely in $\{y\in\bbR^2 \,|\, y\cdot e\le -\ce(u,f)C_0(m^{-1}+\eps^{-2})-2\}$. Then by Theorem \ref{T.3.1}, each connected component of $\Gamma$ has diameter at most $\tfrac 1{10}$ and 
\beq \lb{3.11a}
\Gamma\supseteq \Gamma'\cap \{y\in\bbR^2 \,|\, y\cdot e\le -\ce(u,f)C_0(m^{-1}+\eps^{-2})-3\}.
\eeq
Let $\psi$ solve \eqref{1.3} on the domain $\bbR^2\setminus\Gamma$ with $\psi(0,y)=T(0,y)$ for $y\in \bbR^2\setminus\Gamma$ and $\psi(s,y)=T(s,y)$ ($\ge \xi-\eps$ because $T_t>0$) for $y\in \partial\Gamma$. Then similarly to \eqref{2.3} and \eqref{2.4},
\beq \lb{3.11b}
\psi(t,x) = \bbE_\Om \left( T(0,X^x_t)\chi_{\sigma>t} + T(t-\sigma,X^x_\sigma) \chi_{\sigma\le t} \right) \ge \bbE_\Om \left( T(0,X^x_t)\chi_{\sigma>t} + (\xi-\eps) \chi_{\sigma\le t} \right),
\eeq
with $\sigma=\sigma(\omega)\equiv \inf_{X^x_s(\omega)\in\Gamma} s$, and
\beq \lb{3.11c}
 0\le \psi(t,x) \le T(t,x)\le e^{tm/\zeta} \psi(t,x).
\eeq
If \eqref{3.12} is violated, then \eqref{3.9} and \eqref{3.11} show that 
\[
\bbP_\Om \left(X_t^x\cdot e \le -\ce(u,f)C_0(m^{-1}+\eps^{-2})-3 \right) > \frac{\zeta+\eps}{\xi-\eps},
\]
so that \eqref{3.11a} and \eqref{3.11b} yield $\psi(t,x)>\tfrac{\zeta+\eps}{\xi-\eps}(\xi-\eps) = \zeta+\eps$. But this contradicts \eqref{3.10} and \eqref{3.11c}, so \eqref{3.12} is valid and we are done in the case when \eqref{3.7} holds.

Let us now assume \eqref{3.8}. Here one uses a similar argument with \eqref{3.9}, \eqref{3.10}, and \eqref{3.11} replaced by 
\beq \lb{3.13}
\left| x\cdot e  + \ce(u,f)C_0(m^{-1}+\eps^{-2})+4 \right|\le 1,
\eeq
\beq \lb{3.14}
T(0,x)\ge \xi-\eps,
\eeq
\beq \lb{3.15}
\bbP_\Om \left((X_t^x-x)\cdot e \ge  \alpha_\zeta \sqrt{\tau\diffe(u) } -\tfrac1{10} \right) >  \frac{1-(\xi-\eps)e^{-t m/\zeta}}{1-(\zeta+\eps)}
\eeq
and $\Gamma$  the union of all connected components of the set $\Gamma'\equiv \{y\in\bbR^2\,|\, T(t,y)\ge \zeta+\eps  \}$ lying entirely in $\{y\in\bbR^2 \,|\, y\cdot e\ge \ce(u,f)t + \ce(u,f)C_0(m^{-1}+\eps^{-2})+2\}$. Again each connected component of $\Gamma$ has diameter at most $\tfrac 1{10}$ and 
\beq \lb{3.16}
\Gamma\supseteq \Gamma'\cap \{y\in\bbR^2 \,|\, y\cdot e\ge  \ce(u,f)t + \ce(u,f)C_0(m^{-1}+\eps^{-2})+3\}.
\eeq
If \eqref{3.12} is violated, then \eqref{3.13} and \eqref{3.15} show
\[
\bbP_\Om \left(X_t^x \cdot e \ge  \ce(u,f)t + \ce(u,f)C_0(m^{-1}+\eps^{-2})+3 \right) >  \frac{1-(\xi-\eps)e^{-t m/\zeta}}{1-(\zeta+\eps)},
\]
and so \eqref{3.16} and 
\[
\psi(t,x) = \bbE_\Om \left( T(0,X^x_t)\chi_{\sigma>t} + T(t-\sigma,X^x_\sigma) \chi_{\sigma\le t} \right) \le \bbE_\Om \left( T(0,X^x_t)\chi_{\sigma>t} + (\zeta+\eps) \chi_{\sigma\le t} \right)
\]
(with $\psi$ defined as above) yield
\[
\psi(t,x)< \frac{1-(\xi-\eps)e^{-t m/\zeta}}{1-(\zeta+\eps)}(\zeta+\eps) + \left( 1-\frac{1-(\xi-\eps)e^{-t m/\zeta}}{1-(\zeta+\eps)} \right) = (\xi-\eps) e^{-tm/\zeta}.
\]
This again contradicts \eqref{3.11c}, \eqref{3.14}, and $T_t>0$, so \eqref{3.12} is also valid when \eqref{3.8} holds.
\end{proof}

\section{Quenching by Cellular Flows} \lb{S4}

\begin{proof}[Proof of Theorem \ref{T.1.4}]
It is obviously equivalent to consider the problem on $\bbR\times\bbT$ with $T_0$ supported in $[-L,L]\times\bbT$, which is what we will do. It has been proved in \cite{Kor} that periodic non-degenrate cellular flows in two dimensions have $\diffe(Au)\sim A^{1/2}$ as $A\to\infty$ (the special case of $u_{\rm cell}$ has been treated earlier in \cite{FP}). It is therefore enough to prove that a 1-periodic symmetric non-degenerate cellular flow $u$ quenches solutions of \eqref{1.1} when 
\beq \lb{4.1}
\|f(s)/s\|_\infty\le \gamma_\tht
\eeq
and $T_0$ is supported in $[-b_\tht\sqrt{\diffe(u)}, b_\tht\sqrt{\diffe(u)}]\times\bbT$, with $e=e_1=(1,0)$ and some $b_\tht>0$. Having Theorem \ref{T.2.1} at hand, the proof is similar to that of Lemma~4.2 in \cite{ZlaPersym}. 

Let us prove the last claim. We only need to consider large enough $\diffe(u)$ because solutions supported in $[-b_\tht',b_\tht']$ are quenched for any $u$ provided \eqref{4.1} holds and $b_\tht'>0$ is small enough \cite[Theorem 1.1(ii)]{ZlaArrh}. By the comparison principle, it is sufficient to consider initial data
$T_0(x)\equiv\chi_{[-L,L]}(x_1)$ with $L\equiv \lfloor \tfrac{\tht^2}{64C} \sqrt{\diffe(u)}-\tfrac\tht 2 \rfloor$  and $C$ from Theorem~\ref{T.2.1} (then we can take $b_\tht\equiv \tht^2(128C)^{-1} \min \{1, 2b_\tht'/3 \}$ and have $b_\tht\sqrt{\diffe(u)}\le \max \{b_\tht',L\}$). Let $\psi$ be
the solution of \eqref{1.3} with initial datum $\psi_0\equiv T_0$. We first claim that 
there is a continuous curve $h:[0,1]\to [0,1]\times\bbT$
such that $(h(0))_1=0$ and $(h(1))_1=1$ , and for all $s\in[0,1]$
and $t\ge 1$,
\begin{equation} \lb{4.8}
\psi(t,h(s)) \le \frac \tht 4.
\end{equation}

To this end we let $\phi$ be the solution of \eqref{1.3} with
initial condition $\phi(0,x)\equiv\chi_{[-K-2,K]}(x_1)$ where $K\equiv \lceil
8L\tht^{-1} \rceil$.  Theorem \ref{T.2.1} with 
\[
\alpha\equiv \frac{K+2}{\sqrt{\diffe(u)}} \le \frac{8L+4\tht}{\tht \sqrt{\diffe(u)}} \le \frac{\tht}{8C} 
\]
 shows that there are $\tau\le 1$ and $y\in[-1,0]\times\bbT$ such that
\[
\phi(\tau,y)=\bbP_\Omega \big((X^{y}_\tau)_1\in [-K-2,K] \big) \le
\frac\tht 8.
\]
The maximum principle for \eqref{1.3} implies that the connected
component of the set
\[
\{(t,x)\in[0,\tau]\times \bbR \times \bbT \,|\,\phi(t,x)\le\tfrac\tht 8\}
\]
containing $(\tau,y)$ must intersect
\[
\{x\in \bbR\times\bbT \,|\,\phi(0,x)\le\tfrac\tht 8\} =
(\bbR\setminus[-K-2,K])\times\bbT.
\]
Since by symmetry $\phi(t,x_1,x_2)=\phi(t,-2-x_1,x_2)$ for $x_1\ge
0$, this means that there is a curve $h(s)$ joining
$\{0\}\times\bbT$ and $\{K\}\times\bbT$ such that for each $s$ there
is $\tau_s\le\tau$ with 
\beq \lb{4.3}
\phi(\tau_s,h(s))=\bbP_\Omega \big((X^{h(s)}_{\tau_s})_1\in [-K-2,K] \big) \le \frac\tht 8.
\eeq
Lemma \ref{2.1}(iii) in \cite{ZlaPersym} shows for any $M\in\bbN$, $t\ge 0$, and any 1-periodic symmetric flow,
\[
\bbP_\Omega \big( |X_t^y|\le M \big) \le \left\lceil \frac{|y_1|}{M} \right\rceil^{-1}.
\]
Applying this with $M=L$, $y=X^{h(s)}_{\tau_s}$ and $t-\tau_s$ (for $t\ge 1\ge\tau_s$) in place of $t$, we obtain using \eqref{4.3} for any $t\ge 1$ and $s$,
\[
\psi(t,h(s))=\bbP_\Omega \big(|X^{h(s)}_t|\le L] \big) \le \frac\tht 8 + \Big(1-\frac\tht 8 \Big) \left\lceil \frac{K}{L} \right\rceil^{-1} \le \frac\tht 8 + \Big(1-\frac\tht 8 \Big) \frac\tht 8 \le \frac \tht 4
\]
which is \eqref{4.8} (after reparametrization of $h$ and restriction
to $s\in[0,1]$).

Symmetry of $u$ and $\psi_0$ implies that \eqref{4.8} holds for
$h(s)$ extended to $s\in[-1,1]$ by $h(-s)=(-(h(s))_1,(h(s))_2)$.
Finally, \eqref{4.8} applies to $h(s)$ extended periodically (with
period 2) onto $\bbR$. The last claim holds because
$\psi(t,x)\ge\psi(t,x+(2,0))$ when $x_1\ge -1$ (and
$\psi(t,x)\ge\psi(t,x-(2,0))$ when $x_1\le 1$). This in turn
follows because $\psi(t,x)-\psi(t,x+(2,0))$ solves \eqref{1.3} with
initial datum that is symmetric across $x_1=-1$ and non-negative on
$[-1,\infty)\times\bbT$, and hence stays such by the symmetry of $u$ across $x_1=-1$ (the latter is due to the symmetry of $u$ across $x_1=0$ and periodicity).

This means that $\|\psi(t+1,\cdot)\|_\infty\le
\|\eta(t,\cdot)\|_\infty+ \tfrac \tht 4$ where $\eta$ is the solution of
\eqref{1.3} on the torus $[-1,1]\times\bbT$ (with $-1$ and $1$ identified) with $\eta(0,x)\equiv 1$ and
$\eta(t,h(s))=0$ for $t>0$ and $s\in[-1,1]$.
Then Lemma 2.3 in \cite{ZlaPersym} shows that there is a universal constant $\del>0$ such that
$\|\psi(t+1,\cdot)\|_\infty\le 2e^{-\del t} + \tfrac \tht 4$. We let $\tau_\tht\equiv \tfrac 1\del \ln \tfrac 8\tht +1$ so that $\|\psi(\tau_\tht,\cdot)\|_\infty\le \tfrac \tht 2$. If now $\gamma_\tht\equiv \tau_\tht^{-1}\ln 2$, then 
\[
\|T(\tau_\tht,\cdot)\|_\infty \le e^{\gamma_\tht \tau_\tht} \|\psi(\tau_\tht,\cdot)\|_\infty \le \tht.
\]
The maximum principle then implies that for $t\ge \tau_\tht$ the function $T$ solves \eqref{1.3} with $\|T(t,\cdot)\|_\infty\le\tht$,  and quenching follows.
\end{proof}

\section{A Counterexample in 3D} \lb{S5}

In this last section we show that our main result, Theorem \ref{T.1.1} does not hold in three and more dimensions. The counterexample we provide also shows a {\it breaking of symmetry} in more dimensions. Specifically, despite the fact that $\diffe(u) = D_{-e}(u)$ for any $e$ and $u$, it is not true that $\ce(u,f)c_{-e}^*(u,f)^{-1}$ is bounded by $u$-independent positive constants when $d\ge 3$. 


\begin{theorem} \lb{T.5.1}
Assume $d\ge 3$, $e=e_1$, and $f$ is any KPP reaction.  Then there is no $C_1>0$ such that for any  1-periodic incompressible mean-zero $C^{1,\del}$ flow $u$ on $\bbR^d$ we have $\ce(u,f)\ge C_1\sqrt{\diffe(u)}$, and there is no $C_2<\infty$ such that for any such $u$ we have $\ce(u,f)\le C_2\sqrt{\diffe(u)}$.
\end{theorem}

\begin{proof}
We will use the following two results from \cite{ZlaPercol}:
\begin{align*}
\lim_{A\to\infty} \frac {\ce(Au,f)} A & = \sup_{\substack{w\in\calI \\
\|\nabla w\|_2^2\le f'(0)\|w\|_2^2}} \frac{\int_{\bbT^d}(u\cdot e)w^2\,dx}{\|w\|_2^2}, \\
\lim_{A\to\infty} \frac {\sqrt{\diffe(Au)}} {A} & = \sup_{w\in\calI}
\frac{\int_{\bbT^d}(u\cdot e)w\,dx}{\|\nabla w\|_2},
\end{align*}
where
\begin{equation*}
\calI \equiv \big\{ w\in H^1(\bbT^d) \,\,\,\big|\,\,\,  u\cdot\nabla w = 0 \big\}.
\end{equation*}
We let $e\equiv e_1$ and we will assume $d=3$ since the general case is analogous. Let $\chi\ge 0$ be a smooth characteristic function of the unit disc in $\bbR^2$ and for small $R>0$ let $\chi_R(x_2,x_3)\equiv R^2\|\chi\|_1^{-1} \sum_{j,k\in\bbZ^2} \chi(\tfrac {x_2-j}R,\tfrac{x_3-k}R)$. We will consider the mean-zero periodic shear flows $Au_R(x)=(Au_R(x_2,x_3),0,0)$ with $A\in\bbR$ and $u_R(x_2,x_3)\equiv\chi_R(x_2,x_3)-1$. In this case the elements of $\calI$ are precisely the $w\in H^1(\bbT^3)$ which are independent of $x_1$ and the above formulae become
\begin{align}
\lim_{A\to\infty} \frac {\ce(Au_R,f)} {|A|} & = \sup_{\substack{w\in\calI' \\
\|\nabla w\|_2^2\le f'(0)\|w\|_2^2}} \frac{\int_{\bbT^2}u_Rw^2\,dx_2dx_3}{\|w\|_2^2}, \lb{5.1}\\
\lim_{A\to-\infty} \frac {\ce(Au_R,f)} {|A|} & = \sup_{\substack{w\in\calI' \\
\|\nabla w\|_2^2\le f'(0)\|w\|_2^2}} \frac{\int_{\bbT^2}(-u_R)w^2\,dx_2dx_3}{\|w\|_2^2}, \lb{5.2}\\
\lim_{A\to\pm\infty} \frac {\sqrt{\diffe(Au_R)}} {|A|} & = \sup_{w\in\calI'}
\frac{\int_{\bbT^2}u_Rw\,dx_2dx_3}{\|\nabla w\|_2}, \lb{5.3}
\end{align}
where $\calI' \equiv  H^1(\bbT^2)$ and $e=e_1$.

For $(x_2,x_3)\in[-\tfrac12,\tfrac 12]^2=\bbT^2$ and $r=\sqrt{x_2^2+x_3^2}$ let $w_R(x_2,x_3)=\max\{0,\min\{\log\tfrac 1{2R},\log\tfrac 1{2r}\}\}$. Taking $w\equiv w_R$ we see that RHS of \eqref{5.3}  $\gtrsim \sqrt{\log\tfrac 1R}$ for small $R$.  On the other hand, $-u_R\le 1$ and so RHS of \eqref{5.2} $\le 1$. This proves the first claim.

Now let $\til w_R$ be the maximizer of the RHS of \eqref{5.3}, which exists by \cite{ZlaPercol}, normalized by $\|\nabla \til w_R\|_2=1$ and $\int_{\bbT^2} \til w_R \,dx_2dx_3=0$. The Poincar\' e inequality then gives $\|\til w_R\|_2^2\le \til C$ for some $\til C\ge 1$. Let $C\equiv \max\{\til C,(f'(0))^{-1}\}$, choose $K_R\in\bbR$ so that $\|\til w_R+K_R\|_2^2=C$, and define $w_R\equiv \til w_R+K_R$. Then $1=\|\nabla w_R\|_2^2=\|w_R\|_2^2/C \le f'(0)\|w_R\|_2^2$ (so $w_R$ enters in the RHS of \eqref{5.1}) and $w_R$ also maximizes the RHS of \eqref{5.3}. Schwarz inequality, $\|w_R\|_2^2=C$, $u_R=\chi_R -1$, and the previous paragraph then imply
\begin{equation}\lb{5.4}
\int_{\bbT^2} \chi_Rw_R\,dx_2dx_3 + \sqrt C \ge \int_{\bbT^2}u_Rw_R\,dx_2dx_3 = \text{RHS of \eqref{5.3}} \gtrsim \sqrt{\log\tfrac 1R}
\end{equation}
for small $R$. But then taking $w\equiv w_R$ in \eqref{5.1} and using Schwarz inequality, $\int_{\bbT^2}\chi_R\,dx_2dx_3=1$, and \eqref{5.4} gives 
\[
\text{RHS of \eqref{5.1}}\ge \tfrac 1C \int_{\bbT^2}\chi_Rw_R^2\,dx_2dx_3 -1 \ge \frac { \bigg(\int_{\bbT^2}\chi_Rw_R\,dx_2dx_3\bigg)^2} {C\int_{\bbT^2}\chi_R\,dx_2dx_3}-1 \gtrsim (\text{RHS of \eqref{5.3}})^2 
\]
for small $R$. This together with \eqref{5.4} proves the second claim and we are done.
\end{proof}



\end{document}